\documentclass[a4paper, english,nolineno]{socg-lipics-v2021}

\hideLIPIcs  

\newtheorem{problem}[theorem]{Problem}

\title{From Geometry to Topology: Inverse Theorems for Distributed Persistence} 

\author{Elchanan Solomon}{Duke University Mathematics Department, Durham, USA \and \url{http://www.elchanansolomon.com} }{elchanansolomon@gmail.com}{https://orcid.org/0000-0003-3461-4556}{AFOSR Grant FA9550-18-1-0266.}

\author{Alexander Wagner}{Duke University Mathematics Department, Durham, USA}{alexander.wagner@duke.edu}{https://orcid.org/0000-0002-5961-7852}{NSF CCF-1934964}
\author{Paul Bendich}{Duke University Mathematics Department and Geometric Data Analytics, Inc., Durham, USA}{}{}{AFOSR Grant FA9550-18-1-0266.}

\authorrunning{E. Solomon, A. Wagner, and P. Bendich} 

\Copyright{Elchanan Solomon, Alexander Wagner, and Paul Bendich} 

\ccsdesc[10003744]{Algebraic Topology} 

\keywords{Applied Topology, Persistent Homology, Inverse Problems, Subsampling} 

\category{} 

\relatedversion{A full version of this paper is available at \url{https://arxiv.org/abs/2101.12288}.} 

\acknowledgements{}

\EventEditors{Xavier Goaoc and Michael Kerber}
\EventNoEds{2}
\EventLongTitle{38th International Symposium on Computational Geometry (SoCG 2022)}
\EventShortTitle{SoCG 2022}
\EventAcronym{SoCG}
\EventYear{2022}
\EventDate{June 7--10, 2022}
\EventLocation{Berlin, Germany}
\EventLogo{socg-logo}
\SeriesVolume{224}
\ArticleNo{XX} 

\begin{document}

\maketitle

\begin{abstract}
What is the “right” topological invariant of a large point cloud X?  Prior research has focused on estimating the full persistence diagram of X, a quantity that is very expensive to compute, unstable to outliers, and far from injective. We therefore propose that, in many cases, the collection of persistence diagrams of many small subsets of X is a better invariant. This invariant, which we call “distributed persistence,” is \emph{perfectly parallelizable}, more stable to outliers, and has a rich inverse theory. The map from the space of metric spaces (with the quasi-isometry distance) to the space of distributed persistence invariants (with the Hausdorff-Bottleneck distance) is globally bi-Lipschitz. This is a much stronger property than simply being injective, as it implies that the inverse image of a small neighborhood is a small neighborhood, and is to our knowledge the only result of its kind in the TDA literature. Moreover, the inverse Lipschitz constant depends on the size of the subsets taken, so that as the size of these subsets goes from small to large, the invariant interpolates between a purely geometric one and a topological one. Lastly, we note that our inverse results do not actually require considering all subsets of a fixed size (an enormous collection), but a relatively small collection satisfying simple covering properties. These theoretical results are complemented by synthetic experiments demonstrating the use of distributed persistence in practice.
\end{abstract}

\section{Introduction}
Morphometric techniques in data analysis can be loosely divided into the geometric and the topological. Geometric techniques, like landmarks, the Gromov-Hausdorff metric, optimal transport methods, PCA, MDS \cite{Kruskal:1964aa}, LLE \cite{Roweis2323}, and Isomap \cite{Tenenbaum2319}, are designed to capture some combination of global and local metric structure. Many geometric methods can be solved exactly or approximately via spectral methods, and hence are fast to implement using iterative and sketching algorithms. In contrast, topological techniques, like t-SNE \cite{JMLR:v9:vandermaaten08a}, UMAP \cite{McInnes2018}, Mapper \cite{SPBG:SPBG07:091-100}, and persistent homology, aim to capture large-scale connectivity structure in data. The growing popularity of t-SNE and UMAP as dimensionality reduction  methods suggests that many data sets are topologically, but not metrically, low-dimensional. In this paper, we introduce a new technique into topological data analysis (TDA) that:
\begin{enumerate}
	\item Provably interpolates between topological and geometric structure (Theorem \ref{thm:invstab}).
	\item Is \emph{perfectly parallelizable}.
	\item Is provably stable to perturbation of the data (Proposition \ref{prop:stability}).
	\item Is provably invertible, with globally stable inverse (Theorems \ref{thm:inv}, \ref{thm:invstab}, \ref{thm:sparseinv}, and Corollary \ref{cor:densecover}). Moreover, these inverse results do not require computing the full invariant, but a relatively small subset that can largely be chosen at random (Propositions \ref{prop:probcover} and \ref{prop:probeps}).
	\item Suggests new methods for a host of morphometric challenges, ranging from dimensionality reduction to feature extraction (Section \ref{sec:applications}).
\end{enumerate}

The theoretical guarantees provided here are, to our knowledge, unmatched by any other method in topological data analysis. In addition to these theoretical contributions, we demonstrate our theoretical results empirically on synthetic data sets.


\section{The Distributed Topology Problem}

Let $\lambda$ be an invariant of metric spaces $(X,d_X)$. For $k \in \mathbb{Z}$, we can define a distributed invariant $\Lambda_{k}$ that maps the metric space $X$ to the set of pairs $\{(S,\lambda(S)) \mid S \subset X, |S| = k \}$ if $k > 0$ and to $\emptyset$ otherwise. Put another way, $\Lambda_{k}(X)$ records the values of $\lambda$ on subsets of $X$ of a fixed size. 

When the computational complexity of $\lambda$ scales poorly in the size of $X$, it is much faster to compute $\lambda$ for many small subsets of $X$. $\Lambda_{k}$ takes this intuition to its limit by performing this calculation for \emph{all} subsets of a given size. Although it is unfeasible to actually compute $\Lambda_{k}$ in its entirety, sampling from $\Lambda_{k}$ is simple. This distinguishes $\Lambda_{k}$ from the original invariant $\lambda$, which, in general, cannot be ``sampled from'' or broken into smaller pieces. Moreover, $\Lambda_{k}$ may contain just as much, if not more, information than $\lambda$:
\begin{itemize}
	\item Let $\lambda$ send a finite point cloud $X$ in $\mathbb{R}^d$ to its Euclidean distance matrix. For all $k \geq 2$, $\Lambda_{k}$ contains the same information as $\lambda$.
	\item Let $\lambda$ send a finite point cloud $X$ in $\mathbb{R}^d$ to its diameter. For any $k \geq 2$, $\Lambda_{k}$ can be used to deduce $\lambda$. 
	\item Let $\lambda$ send a finite point cloud $X$ in $\mathbb{R}^d$ to its mean. For any $k \geq 1$, $\Lambda_{k}$ can be used to deduce $\lambda$. In fact, if $k < |X|$, $\Lambda_{k}$ determines $X$ up to rigid motion.
\end{itemize} 

Finally, $\Lambda_{k}$ is more robust than $\lambda$, as outliers in $X$ have no impact on $\Lambda_{k}$ for outlier-free subsets $S \subset X$. The theoretical goal of this paper is to address the following questions:
\begin{problem}
	If $\lambda$ is a topological invariant of metric spaces, how much information is contained in $\Lambda_{k}$ for various $k$? Does $\Lambda_{k}$ determine $\lambda$, or perhaps contain strictly more information?
	\label{prob:prob1}
\end{problem}

\begin{problem}
	How does the information contained in $\Lambda_{k}$ depend on the parameter $k$?
	\label{prob:prob2}
\end{problem}

\begin{problem}
	What information can be deduced from $\Lambda_{k}$ if we can only compute it for a relatively small collection of subsets?
	\label{prob:prob3}
\end{problem}

\subsection{Case Study: The Noisy Circle}
To illustrate the advantage of working with distributed invariants, we compare three data sets of $500$ points. The first is spaced regularly around a circle, the second sampled uniformly from the unit disc, and the third contains $450$ points on the circle and $50$ points sampled from the disc (we call this the \emph{noisy circle}), see Figure \ref{fig:noisycircle}. For each of these point clouds, we compute their full $1$-dimensional persistence diagrams, see Figure \ref{fig:noisycircle_diag}. In addition, for each point cloud, we sample $1000$ subsets of size $10$, compute the resulting $1000$ $1$-dimensional persistence diagrams, vectorize them as \emph{persistence images}\footnote{This is a technique for turning a persistence diagram into a function by placing a Gaussian kernel at each dot in the persistence diagram, with mean and variance varying by location, cf. \cite{adams2017persistence}.}, and average the results, see Figure \ref{fig:noisycircle_img}. The persistence diagram of the noisy circle is most similar to that of the disc (in Bottleneck distance), demonstrating that ordinary persistence does not see the circle around which most of the data points are clustered. The distributed persistence, however, tells a different story. The distribution for the noisy circle interpolates between the distributions of the other two spaces, but is substantially closer to that of the circle than the disc.

\begin{figure}[htb!]
	\centering
	\includegraphics[scale=0.4]{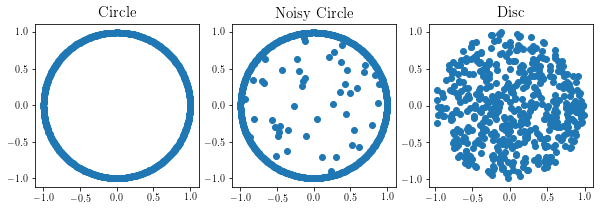}
	\caption{Three point clouds: the circle, the noisy circle, and the disc.}
	\label{fig:noisycircle}
\end{figure}

\begin{figure}[htb!]
	\centering
	\includegraphics[scale=0.6]{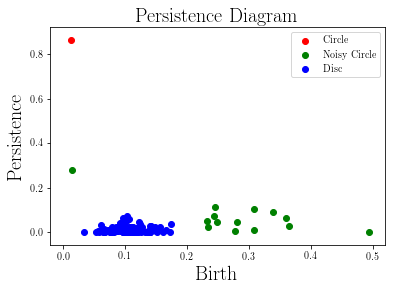}
	\caption{The persistence diagrams of our point clouds, plotted in birth-persistence coordinates.}
	\label{fig:noisycircle_diag}
\end{figure}

\begin{figure}[htb!]
	\centering	\includegraphics[scale=0.6]{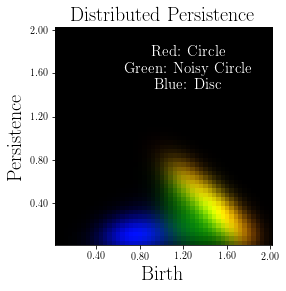}
	\caption{Averaged distributed persistence images of our three spaces. The dominant orange/yellow region is the overlay of the circle (red) distribution and the noisy circle (green) distribution.}
	\label{fig:noisycircle_img}
\end{figure}

\section{Prior Work on Distributed Topology}

In \cite{pmlr-v37-chazal15}, Chazal et al. propose the following framework. Given a metric measure space $(\mathbb{X}, \rho, \mu)$, sample $m$ points and compute the persistence landscape of the associated Vietoris-Rips filtration. This procedure produces a random persistence landscape, $\lambda$, whose distribution is denoted $\Psi_{\mu}^m$. Repeating this procedure $n$ times and averaging produces the empirical average landscape, an unbiased estimator of the average landscape $E_{\Psi_{\mu}^m}[\lambda]$. This approach is similar to the distributed topological invariants considered in this paper, except we consider a collection of topological invariants as a labeled set rather than taking their sum. Though Bubenik \cite{10.1007/978-3-030-43408-3_4} gives conditions in Theorem $5.11$ under which a collection of persistence diagrams may be reconstructed from the average of their corresponding persistence landscapes, such an inverse exists only generically, and is highly unstable.


The main theorem of \cite{pmlr-v37-chazal15} is that the average landscape is stable with respect to the underlying measure. Specifically, if $\mu$ and $\nu$ are two probability measures on the same metric space $(\mathbb{X}, \rho)$, the sup norm between induced average landscapes is bounded by $m^{1/p}W_{\rho, p}(\mu, \nu)$ for any $p \geq 1$. Similar results were obtained in \cite{Blumberg:2014aa} for distributions of persistence diagrams of subsamples. In particular, Blumberg et al. showed that the distribution of barcodes with the Prokhorov metric is stable with respect to the associated compact metric measure space in the Gromov-Prokhorov metric. Both results are analogous to the stability of the distributed topological invariants given in Proposition \ref{prop:stability}. However, working with labeled collections of distributed topological invariants, we are also able to provide inverse stability results, such as our main Theorem \ref{thm:invstab}, which states that changes in the metric structure are bounded with respect to changes in the distributed topological invariants.  

In \cite{memoli2012some}, M\'{e}moli developed the study of \emph{curvature sets}, an invariant introduced by Gromov that consists of computing the distance matrices of every subset of a fixed size in an ambient metric space. Shortly after this paper appeared, G\'{o}mez and M\'{e}moli \cite{gomez2021curvature} released a manuscript studying the collection of persistence diagrams of subsets of bounded cardinality in an ambient metric space. This construction is similar to ours, with the following key differences: firstly, we take subsets of a fixed cardinality $k$, or else cardinalities in a small neighborhood of $k$, whereas G\'{o}mez and M\'{e}moli consider all subsets of cardinality at most $k$, and, secondly, we have different conventions for which homological degrees to compute. More importantly, that paper differs from ours in the nature of the results: G\'{o}mez and M\'{e}moli are focused on computing this invariant for simple spaces, and giving examples of when their invariant characterizes the homotopy type of the underlying space. This paper is focused on inverse results of a geometric flavor, trying to understand how distributed topological invariants characterize the quasi-isometry type of the underlying space.

In \cite{Bubenik_2020}, Bubenik et al. consider unit disks, denoted $D_K$, of surfaces of constant Gaussian curvature $K$ with $K \in [-2, 2]$. Since these spaces are all contractible, their reduced singular homology is trivial and global homology cannot distinguish them. However, the authors prove that the maximum \v{C}ech persistence for three points sampled from $D_K$ determines $K$. The authors also successfully apply the same empirical framework of average persistence landscapes from \cite{pmlr-v37-chazal15} to experimentally determine the curvature of $D_K$ for various $K$. The authors in \cite{PhysRevE.93.052138} used average persistence landscapes to provide experimental verification of a known phase transition. Finally, the authors in \cite{10.1007/978-3-030-42266-0_14} use average persistence landscapes to achieve improved results, compared to standard machine learning algorithms, in disease phenotype prediction based on subject gene expressions.

\section{Background}
The content of this paper assumes familiarity with the concepts and tools of persistent homology. Interested readers can consult the articles of Carlsson \cite{carlsson2009topology} and Ghrist \cite{ghrist2008barcodes} and the textbooks of 
Edelsbrunner and Harer \cite{edelsbrunner2010computational} and Oudot \cite{oudot2015persistence}. We include the following primer for readers interested in a high-level, non-technical summary.

Persistent homology records the way topology evolves in a parametrized sequence of spaces. To apply persistent homology to a metric space, a pre-processing step is needed that converts the metric space into such a sequence. The two classical ways of doing this are called the Rips and \v{C}ech filtrations, respectively; the former is much easier to compute than the latter, but contains less geometric information. Both consist of inserting simplices into the metric space at a parameter value equal to the proximity of the associated vertex points. As the sequence of spaces evolves, the addition of certain edges or higher-dimensional simplices changes the homological type of the space -- these simplices are called critical. Persistent homology records the parameter values at which critical simplices appear, notes the dimension in which the homology changes, and pairs critical values by matching the critical value at which a new homological feature appears to the critical value at which it disappears. This information is organized into a structure called a persistence diagram, and there are a number of metrics with which persistence diagrams can be compared. 

If one forgets about the pairing and retains only the dimension information of the critical values, the resulting invariant is called a Betti curve. Betti curves are simpler to compute and work with than persistence diagrams, but are less informative and harder to compare. Finally, if one also drops the dimension information by taking the alternating sum of the Betti curves, one gets an Euler curve. Euler curves are even less discriminative than Betti curves, but enjoy the special symmetry properties of the Euler characteristic. These symmetries will be put to good use in this paper.

Persistence theory guarantees that a small modification to the parametrization of a sequence of spaces implies only small changes in its persistence diagram. To be precise, if the appearance time of any given simplex is not delayed or advanced by more than $\epsilon$, the persistence diagram as a whole is not distorted by more than $\epsilon$ in the appropriate metric (called the \emph{Bottleneck distance}). Throughout this paper we will use the trick of modifying filtrations by rounding their critical values to a fixed, discrete set.

As a rule, the map sending a metric space to its persistence diagram is not injective, as many different point clouds share the same persistence diagram \cite{curry2018fiber,leygonie2021fiber,leygonie2021algorithmic}. Moreover, the set of metric spaces sharing a common persistence diagram need not be bounded, so that arbitrarily distinct spaces might have the same persistence. There are a number of constructions in the TDA literature that attempt to correct this lack of injectivity by constructing more sophisticated invariants; these are often called \emph{topological transforms}. Examples include the Persistent Homology Transform \cite{turner2014persistent,ghrist2018persistent,curry2018many,kirveslahti2021representing} and Intrinsic Persistent Homology Transform \cite{oudot2017barcode}; consult \cite{oudot2020inverse} for a survey of inverse results in persistence. These methods are largely infeasible to approximate, unstable, and provide no global Lipschitz bounds on their inverse, so two wildly different spaces may produce arbitrarily similar (though not exactly identical) transforms. The distributed topology invariant studied in this paper is injective, easy to sample from, stable, and with Lipschitz inverse. 

We conclude with an analysis of the computational complexity of persistence calculations. Persistence calculations are $O(N^{\omega})$, where $N$ is the number of simplices in the complex and $\omega$ is the matrix multiplication constant \cite{milosavljevic2011zigzag}. For a metric space $X$, the number of $(d+1)$-dimensional simplices in the Rips complex is ${|X| \choose d+2}$, which are needed for computing persistence in degree $d$. Thus the computational complexity is $O({|X| \choose d+2}^{\omega})$, which is huge even for small values of $d$. Computing persistence of $M$ $k$-element subsets is $O(M{k \choose d+2}^{\omega})$, which is orders of magnitude smaller for the values of $M$ used in the experiments of Section \ref{sec:experiments}.


\section{Theoretical Results}

In what follows, we let $\lambda$ be any of the following four topological invariants: (1) Rips Persistence (RP), (2) Rips Euler Curve (RE), (3) \v{C}ech Persistence (CP), and (4) \v{C}ech Euler Curve (CE). To be precise, RP and CP consist of persistence diagrams for every homological degree. When working with either of these invariants, the Bottleneck or Wasserstein distance is the maximum of the Bottleneck or Wasserstein distances over all degrees. Our decision to focus on these four invariants is motivated by a desire to keep the following analysis as simple and concrete as possible, and many of the arguments and theorems below carry through, with minor modification, for other invariants. Indeed, a very similar analysis works for functional persistence, where the sampling consists of picking $k$ points at random and computing functional persistence on their induced subcomplex; details of this proof will appear in future work.

\subsection{Stability}

A result of the following form is standard in the TDA literature, and demonstrates the ease of producing stable invariants using persistent homology.

\begin{definition}
	\label{def:qi}
	Let $(X, d_X)$ and $(Y, d_Y)$ be metric spaces. A map $\phi:(X,d_X) \to (Y,d_Y)$ is an $\epsilon$-quasi-isometry if $|d_X(x_1,x_2) - d_Y(\phi(x_1),\phi(x_2))| \leq \epsilon$ for all $x_1, x_2 \in X$. The quasi-isometry distance between $X$ and $Y$ is the smallest $\epsilon$ for which such a map exists.
\end{definition}

\begin{proposition}
	\label{prop:stability}
	Let $\phi:(X,d_X) \to (Y,d_Y)$ be an $\epsilon$-quasi-isometry of metric spaces. Then for all subsets $S \subseteq X$, and $\lambda$ either RP or CP, $d_{B}(\lambda(S),\lambda(\phi(S))) \leq \epsilon$, where $d_{B}$ is the Bottleneck distance on persistence diagrams.
\end{proposition}
\begin{proof}
	This follows immediately from the Gromov-Hausdorff stability theorem for persistence diagrams of metric spaces \cite{chazal2016structure,cohen2007stability}.
\end{proof}

\subsection{Injectivity}

In this section, we show how distributed persistence can be used to reconstruct the isometry type of a metric space. This provides an answer to Problem \ref{prob:prob1}. To help motivate this result, we consider the simple cases of $k=2$ and $k=3$.

\begin{lemma}
	\label{lem:2dist}
	For all of our invariants, $\Lambda_{2}$ determines the isometry type of $X$, and hence also $\lambda(X)$.
\end{lemma}
\begin{proof}
	For each invariant, the distance between two points $x,y \in X$ can be read off of $\lambda({x,y})$, so $\Lambda_{2}$ records all the pairwise distances between points in $X$, and hence the metric $d_{X}$. The metric then determines the Rips or \v{C}ech complex of $X$ as an abstract metric space. When considering the \v{C}ech complex of a point cloud $X$ in Euclidean space, the metric determines the embedding of $X$ up to Euclidean isometry (see \cite{schoenberg1935remarks}), and hence the \v{C}ech filtration.
\end{proof}


Setting $k=3$ is sufficient to break the implication of an isometry.

\begin{lemma}
	$\Lambda_{3}$ does not determine the isometry type of $X$.
\end{lemma}
\begin{proof}
	A simple counterexample suffices. Let $X$ consist of the vertices of an obtuse triangle with angle $\theta > \pi/2$. Varying the angle $\theta$ in $(\pi/2, \pi)$ alters the isometry type of $X$, but leaves its persistent homology unchanged.
\end{proof}

To obtain stronger results, we introduce the following two generalizations, one to the notion of distributivity, and the other to the invariants $\lambda$.

\begin{definition}
	Let $\mathbf{k} = \{k_1,k_2, \cdots, k_r\}$ be a set of positive integers . We write $\Lambda_{\mathbf{k}}$ for the union $\bigcup_{i=1}^{r}\Lambda_{k_r}$.
\end{definition}

\begin{definition}
	For any of our four invariants $\lambda$, let $\lambda^{m}$ be the invariant restricted to the $m$-skeleton of the Rips or \v{C}ech complex, and define $\Lambda_{\mathbf{k}}^{m}$ analogously.
\end{definition}

Setting $m=0$ provides information only on the cardinality of $X$. The $1$-skeleton contains both geometric and topological information, and its persistence is fast to compute. As $m$ increases, computational complexity goes up, and the resulting invariants record higher-dimensional topological information. The following results demonstrate how knowing sufficiently many Euler characteristic invariants allows one to determine new ones.

\begin{definition}
For a set $X$, let $K(X)$ be the full simplicial complex on $X$, which is abstractly equal to the power set of $X$. A function $f: K(X) \to \mathbb{R}$ on the simplices of $K$ is called \emph{monotone} if $f(\sigma) \leq f(\tau)$ when $\sigma$ is a face of $\tau$. For a subcomplex $T \subseteq K(X)$, we write $T(r)$ to denote the $r$-sublevel set of $f$ on $T$.
\end{definition}

\begin{lemma}
\label{lem:ietrick} Let $R, T_1, \cdots, T_n$ be subcomplexes of $K(X)$, the full complex on $X$. Writing $S^{m}$ to denote the $m$-skeleton of a subcomplex $S$, suppose that $R^{m} = \bigcup_{i=1}^{n} T_i^{m}$. For $f: K(X) \to \mathbb{R}$ a monotone function on $K$, we have:
\begin{align*}
	\chi(R^{m}(r)) &= \chi\left(\bigcup_{i} T_i^{m}(r)\right) \\ 
	&= \sum_{i} \chi(T_i^{m}(r)) \\
	&- \sum_{i < j}\chi(T_i^{m}(r) \cap T_j^{m}(r)) \\
	&+ \sum_{i < j < k}\chi(T_i^{m}(r) \cap T_j^{m}(r) \cap T_k^{m}(r)) \,\,\,\,\, \cdots \\
	&+ (-1)^{n-1} \chi(T_1^{m}(r) \cap \dots \cap T_{n}^{m}(r)).
\end{align*}	
\end{lemma}
\begin{proof}
This follows from the inclusion-exclusion property of the Euler characteristic.
\end{proof}

\begin{lemma}
	Let $\lambda$ be RE or CE. For any metric space $X$ and $k \geq m+2$, let $\mathbf{k} = \{k,k-1,\cdots, k-m-1\}$. Then $\Lambda_{\mathbf{k}}^{m}$ determines $\Lambda_{k-m-2}^{m}$. 
	\label{lem:euler}
\end{lemma}
\begin{proof}
	Let $Y \subset X$ be a subset of size $(k-m-2)$. Let $\{x_1, \cdots, x_{m+2}\}$ be points in $X \setminus Y$, and set $W = Y \cup \{x_1, \cdots, x_{m+2}\}$ and $W_i = W \setminus \{x_i\}$. Let $f: K(X) \to \mathbb{R}$ be the function giving rise to the Rips or \v{C}ech filtration, and let $R = K(W)$ and $T_i = K(W_i)$. By construction, $R^{m} = \bigcup_{i=1}^{n} T_i^{m}$, since every $(m+1)$-element subset of $W$ lies in some $W_i$, so we may apply Lemma \ref{lem:ietrick}. This gives a formula for the Euler characteristic of $R^{m}(r)$ in terms of the Euler characteristics of the $T^{m}_{i}(r)$ and their intersections. By hypothesis, we know the Euler characteristics of every term in this equation other than the final term,  $\chi(T_1^{m}(r) \cap \dots \cap T_{n}^{m}(r)) = \chi(K^{m}(Y)(r))$, since every other term involves the Euler characteristic of a set with cardinality in $\mathbf{k}$. This means that we can solve for $\chi(K^{m}(Y)(r))$ in terms of known quantities, and hence deduce the Euler curve for the Rips or \v{C}ech filtration on $Y$. See Figure \ref{fig:incexc} for a concrete example.
\end{proof}

\begin{figure}[htb!]
	\centering
	\includegraphics{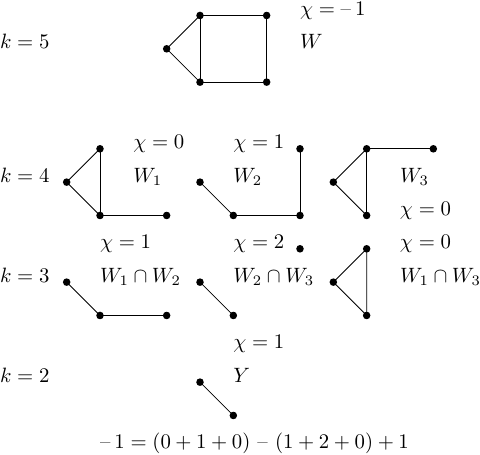}
	\caption{Our goal is to deduce the Euler Characteristic (at a fixed scale $r$) of $Y$, a $1$-simplex consisting of $k=2$ points. This can be derived from the Euler Characteristics of the other subcomplexes in the diagram above.}
	\label{fig:incexc}
\end{figure}

\begin{corollary}
	\label{cor:eulertrick}
	Let $\lambda$ be RE or CE. For any metric space $X$ and $k \geq m+2$, let $\mathbf{k} = \{k,k-1,\cdots, k-m-1\}$. Then $\Lambda_{\mathbf{k}}^{m}$ determines $\Lambda_{2}^{m}$. 
\end{corollary}

\begin{proof}
	Lemma \ref{lem:euler} shows that $\{\Lambda_{k}^{m}, \Lambda_{k-1}^{m}, \cdots \Lambda_{k-m-1}^{m}\}$ determines $\Lambda_{k-m-2}^{m}$. By the same logic, $\{\Lambda_{k-1}^{m}, \Lambda_{k-2}^{m}, \cdots \Lambda_{k-m-2}^{m}\}$ determines $\Lambda_{k-m-3}^{m}$. Repeating this argument, we can deduce $\Lambda_{2}^{m}$.  
\end{proof}

Leveraging Lemma \ref{lem:euler}, we prove that all of our persistence invariants determine the isometry type of $X$.

\begin{theorem}
	For any of the four invariants $\lambda$ and $k \geq m+2 > 2$, let $\mathbf{k} = \{k,k-1,\cdots, k-m-1\}$. Then $\Lambda_{\mathbf{k}}^{m}$ determines the isometry type of $X$.
	\label{thm:inv}
\end{theorem}
\begin{proof}
	When $m \geq 1$, the $m$-skeleton contains all edges in $X$, so Lemma \ref{lem:2dist} applies. If the set $\{k,k-1, \cdots, k-m-1\}$ contains $2$, this follows from Lemma \ref{lem:2dist}. Otherwise, let us assume $\lambda$ is either RE or CE, as RP or CP contain more information than their Euler characteristic counterparts. By Corollary \ref{cor:eulertrick}, we can determine $\Lambda_2^m$ and then apply Lemma \ref{lem:2dist}.
\end{proof}

\begin{remark}
	Note that $m=1$ suffices to apply the prior theorem. As $m$ gets larger, more topological information is needed to determine the isomety type of the underlying space.
\end{remark}


%

\subsection{Inverse Stability}

We now consider what happens if two metric spaces have distributed invariants which are similar but not identical. We show that this implies a quasi-isometry between $X$ and $Y$, with constant depending quadratically on the subset size parameter $k$. This provides a precise answer to Problem \ref{prob:prob2} on how the distributed invariant interpolates between geometry and topology.

The key insight in the proof of this result is that there is always a way to modify the Rips or \v{C}ech filtrations on $X$ and $Y$ to force their distributed invariants to coincide exactly. Taken together with the telescoping trick of Corollary \ref{cor:eulertrick}, this modified invariant must agree for all subsets of size two. Persistence stability allows us to assert that the modified invariant and the original persistence invariant are a bounded distance apart, so equality of the modified invariant gives near-equality of the Rips or \v{C}ech persistences on subsets of size two, which is nothing more than pairwise distance data.

The proposed modification to our filtration consists of rounding it to a discrete set of values. The following technical lemma shows how to pick a rounding set $R$ that aligns two sets of real values without moving any value more than a bounded amount. The proof of this lemma can be found in Appendix \ref{sec:roundinglemmaproof}.
\begin{lemma}[Rounding Lemma]
	Let $P = \{p_1 \leq p_2 \leq \cdots p_N\}$ and $Q = \{q_1, q_2 \cdots, q_N\}$ be two multisets of real numbers. Define $d_i = |p_i - q_i|$, let $\epsilon = \max d_i$ and $\delta = \sum_{i=1}^{n} d_i$. Then there  exists a subset $R \subset \mathbb{R}$ and a map $\pi: P \cup Q \to R$ sending a real value $x$ to the unique closest element in $R$ (rounding up at midpoints), with:
	\begin{enumerate}
		\item $\pi(p_i) = \pi(q_i)$ for all $i$.
		\item $|\pi(x) - x| \leq 3\epsilon + 4\delta$.
	\end{enumerate}
	In particular, since $\epsilon \leq \delta$, we can replace (2) with (2*) $|\pi(x) - x| \leq 7 \delta$.
	\label{lem:roundinglemma}
\end{lemma}

This lemma is central to the proof of the central result of this section, Theorem \ref{thm:invstab}, the details of which can be found in Appendix \ref{sec:majorproof}. The preceding definitions clarify the statement of the theorem:

\begin{definition}
	Let $m < k$ be natural numbers. We define the following partial sum of binomial coefficients:
	\[S(k,m) = {k \choose 2} + {k \choose 3} +\cdots +{k \choose m+1}.\]
\end{definition}

\begin{definition}
	Let $(K,f)$ be a filtered simplicial complex, i.e. a simplicial complex $K$ with a monotone function $f:K \to \mathbb{R}$ encoding the appearance times of simplices. Given a subset $R \subset \mathbb{R}$, \emph{rounding this filtration to $R$} consists of post-composing $f$ with the map sending every element of $\mathbb{R}$ to its nearest element in $R$ (rounding up at midpoints).
\end{definition}

\begin{remark} The appearance time of simplices in an $R$-rounded filtration occur only at values contained in $R$. The effect of this rounding on the resulting persistence diagrams is to round the birth and death times of its constituent dots; no new points are introduced.
\end{remark}


%

\begin{theorem}
	Let $\lambda$ be either RP or CP, and take $k > m > 0$. Let $Z$ and $Y$ be metric spaces, $\phi:Z \to Y$ a map of sets, and $X \subseteq Z$ a subspace such that $\phi_{X}:X \to Y$ is a surjection. Let $\Gamma \subset P(Z)$ be a collection of subsets of cardinality $k$ through $k-m-1$ satisfying the following two properties:
	\begin{itemize}
		\item ($X$-Covering property) For every pair of points $\{x_1,x_2\}$ in $X$ there is a subset $S \in \Gamma$ such that $|S| = k$ and $\{x_1,x_2\} \subset S$.
		\item (Closure property) If $S \in \Gamma$ has $|S| = k$, and $S' \subset S$ has $|S'|\geq k-m-1$, then $S' \in \Gamma$.
	\end{itemize} 
	Suppose that $d_{B}(\lambda^{m}(S),\lambda^{m}(\phi(S))) \leq \epsilon$ for all $S \in \Gamma$. If $\lambda$ is RP, $\phi_{X}$ is a $112k^2\epsilon$ quasi-isometry, and if $\lambda$ is CP, $\phi_{X}$ is a $224S(k,m)\epsilon$ quasi-isometry.
	\label{thm:invstab}
\end{theorem}

\begin{remark}
The collection $\Gamma$ of \emph{all} subsets of size $k$ through $k-m-1$ enjoys the covering and closure properties above. However, it is easy to find much smaller collections satisfying the conditions of Theorem \ref{thm:invstab}, see Section \ref{sec:prob}.
\end{remark}

\begin{remark}
	Theorem \ref{thm:invstab} answers Problem \ref{prob:prob2} by showing that smaller values of $k$ give more control of quasi-isometry type than larger values. This justifies our claim that distributed topology interpolates between local geometry and global topology.
\end{remark}


One can shrink the collection $\Gamma$ further by asking only that its elements cover sufficiently close approximations for $X$ and $Y$; in this case, the resulting bound is not in the quasi-isometry distance but in the Gromov-Hausdorff distance.

\begin{corollary}
	Let $\lambda$ be either RP or CP, and take $k > m > 0$. Let $\phi: Z \to Y$ be a map of metric spaces, $X \subset Z$ a subspace, and $X' \subset Z$ another, potentially much smaller, subspace with $d_{GH}(X,X')<\delta$. Suppose also that $d_{GH}(\phi(X'),Y) < \delta$. Finally, let $\Gamma \subset P(Z)$ be a collection of subsets of cardinality $k$ through $k-m-1$ satisfying the covering and closure properties for $X'$, and such that $d_{B}(\lambda^{m}(S),\lambda^{m}(\phi(S))) \leq \epsilon$ for all $S \in \Gamma$. If $\lambda$ is RP, then $d_{GH}(X,Y) \leq 112 k^2\epsilon + 2\delta$, and if $\lambda$ is CP, then $d_{GH}(X,Y) \leq 224S(k,m)  k^{m+1}\epsilon + 2\delta$.
	\label{cor:densecover}
\end{corollary}
\begin{proof}
	Theorem \ref{thm:invstab} implies that $\phi$ is a quasi-isometry from $X'$ to $\phi(X')$. We can turn this into a a Gromov-Hausdorff matching between $X$ and $Y$ using the facts that $d_{GH}(X,X') < \delta$ and $d_{GH}(\phi(X'),Y) < \delta$, and two applications of the triangle inequality increase the bound by $2\delta$.
\end{proof}

\begin{corollary}
	If $X \subset \mathbb{R}^{d_1}$ and $Y \subset \mathbb{R}^{d_2}$, then the quasi-isometry bound for \v{C}ech persistence in Theorem \ref{thm:invstab} can be replaced with:
	\[112 k^2\left(\epsilon + \sqrt{\frac{2d_1}{d_1 + 1}} + \sqrt{\frac{2d_2}{d_2 + 1}} \right)\]
	Note that the added terms sum at most to $2\sqrt{2}$, so that this bound is better than the bound given in Theorem \ref{thm:invstab} for large values of $\epsilon$, but does fail to go to $0$ as $\epsilon \to 0$.
	\label{cor:cechrips} 
\end{corollary}
\begin{proof}
	The Rips and \v{C}ech persistence of point clouds in $\mathbb{R}^d$ are always within $\sqrt{\frac{2d}{d+1}}$ of one another in the bottleneck distance, cf. Theorem 2.5 in \cite{de2007coverage}. The result then follows by replacing \v{C}ech persistence with Rips persistence and using the triangle inequality.
\end{proof}

\subsection{Topology + Sparse Geometry}
Our goal now is improve the results of the prior section by giving quasi-isometry bounds that scale linearly in $k$, rather than quadratically. This can be accomplished by using an inclusion-exclusion argument on the $1$-skeleton persistence of $X$ that uses only subsets of size $k$ and $k-1$, and does not need subsets of size $k-2$. Namely, given a subset $Y \subset X$ with $|Y| = (k-2)$, we take $Y = W_1 \cap W_2$ for $|W_1| = |W_2| = (k-1)$ and $W = (W_1 \cup W_2)$ with $|W| = k$, as shown in Figure \ref{fig:sparseIE}, and attempt to deduce the Euler characteristic of $Y$ from those of $W_1,W_2$, and $W$. However, the union of the $1$-skeleton complexes on $W_1$ and $W_2$ is not the $1$-skeleton complex on $W$, owing to the fact that $W$ contains an extra edge connecting the pair of vertices in $W \setminus Y$. Indeed, this is why we chose to cover $W$ with \emph{three} subsets of cardinality $k-1$ in Lemma \ref{lem:ietrick}.

\begin{figure}[htb!]
	\centering
	\includegraphics{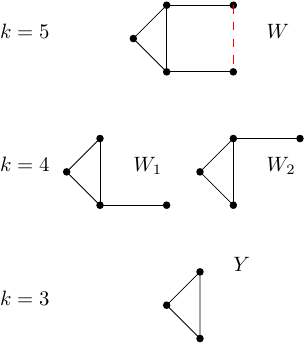}
	\caption{Our goal is to deduce the Euler Characteristic (at a fixed scale $r$) of $Y$, a subcomplex of size $k=3$, using subcomplexes of size $k=4$ and $k=5$. However, the inclusion-exclusion argument fails because the union of the complexes of $W_1$ and $W_2$ is not the complex on $W = W_1 \cup W_2$, and the missing edge is shown in red.}
	\label{fig:sparseIE}
\end{figure}

The effect of this extra edge on persistence is quite subtle, but its effect on the Euler curve is trivial, as it amounts to subtracting a step function supported on $[r, \infty)$, where $r$ is the appearance time of the extra edge in the complex. If we knew $r$, we could correct the deficit in our inclusion-exclusion argument. Note that the we have the freedom to choose $W_1$ and $W_2$ as we like, so to make this argument work we need only know the length of a single edge in $X$ that does not intersect $Y$. A very small collection of edge lengths suffice to patch up the inclusion-exclusion argument for all subsets of $X$ of size at most $k$. The following theorem improves on the bounds in Theorem \ref{thm:invstab} by assuming that $\phi$ is already known to be a quasi-isometry on a sparse subset $L \subset Z$. The proof can be found in Appendix \ref{sec:sparseproof}.

\begin{theorem}
	Let $\lambda, \phi,Z,Y$, and $X$ be as in the statement of Theorem \ref{thm:invstab}, and let $k>m=1$. Let $L \subset Z$ be a subset satisfying the following geometric condition:
		\[\sum_{(x_i,x_j) \in L \times L} |\|x_i - x_j\| - \|\phi(x_i) - \phi(x_j)\|| \leq \epsilon_2.\]
	Let $\Gamma \subset P(Z)$ be a collection of subsets of cardinality in $\{k,k-1\}$ satisfying the following two properties:
	\begin{itemize}
		\item ($(X,L)$-Covering property) For every pair of points $\{x_1,x_2\}$ in $X$, not both in $L$, there is a subset $S \in \Gamma$ such that $|S| = k$, $\{x_1,x_2\} \subset S$, and $S \setminus \{x_1, x_2\} \subset L$.
		\item (Closure property) If $S \in \Gamma$ has $|S| = k$, and $S' \subset S$ has $|S'| = k-1$, then $S' \in \Gamma$.
	\end{itemize} 
	Finally, suppose that  $d_{B}(\lambda^{1}(S),\lambda^{1}(\phi(S))) \leq \epsilon_1$ for all $S \in \Gamma$. Then $\phi_{X}$ is a $56(k+1)\epsilon_1 + 28\epsilon_2$ quasi-isometry.
	\label{thm:sparseinv}
\end{theorem}

\begin{remark}
Relatively few subsets of cardinality $k$ are needed to satisfy the $(X,L)$-covering property, as one subset is needed for every pair of points in $(X \setminus L)$, of which there are ${ |X \setminus L| \choose 2}$, and $S = L \cup \{x\}$ works to cover all pairs of the form $(l,x)$ for $l \in L$ and $x \in X$, adding $|X \setminus L|$ more subsets. Finally, to satisfy the closure property, we include all $(k-1)$-element subsets of these sets, which multiplies the total number of subsets by at most $(k+1)$.	
\end{remark}

\subsection{Probabilistic Results}
\label{sec:prob}
Theorems \ref{thm:invstab} and \ref{thm:sparseinv} and Corollary \ref{cor:densecover} tell us that we do not need to consider all ${|X| \choose k} + {|X| \choose k-1} + \cdots + {|X| \choose k-m-1}$ subsets $S \subseteq X$ of size $|S| \in \{k, \cdots, k-m-1\}$, so long as the collection $\Gamma$ of subsets considered satisfies appropriate cover and closure properties. This still leaves the question of how to produce such a collection $\Gamma$ in practice. Of the two conditions, the covering property is the more flexible, as the closure property explicitly requires the full downward closure of the appropriate cardinalities. The aim of this section is to show that a relatively small collection of randomly chosen size-$k$ subsets are likely to satisfy the covering property, and hence generate a collection $\Gamma$ that is both covering and closed. We will assume that $Z=X$ in the language of Theorem \ref{thm:invstab}, i.e. that we are randomly sampling from the space we wish to cover. All proofs can be found in Appendix \ref{sec:probproof}.\\

The following two propositions, with $p=2$, provide a lower bound on the probability that a random collection of $M$ subsets covers pairs in $X$.

%

\begin{proposition}
	Let $X$ be a set of size $n$, and choose $M$ subsets $\{S_1, \cdots, S_M\}$ of size $k$ by uniform sampling without replacement. Let $p \leq k$ and $A$ be the outcome that every set of $p$ points $(x_1, \cdots, x_p)$ is contained in at least one $S_i$. Then
	\[P(A) \geq 1 - {n \choose p}\left(1 - \left(\frac{k-p+1}{n-p+1}\right)^{p}\right)^{M}.\]
	\label{prop:probcover}
\end{proposition}

\begin{proposition}
	Let $A$ be as in the prior proposition. For any $\epsilon \in (0,1)$, if
	\[M \geq (p \log\left(\frac{ne}{p}\right) - \log(1- \epsilon))\left(\frac{n-p+1}{k-p+1}\right)^{p}\]
	
	then $P(A) \geq \epsilon$.
	\label{prop:probeps}
\end{proposition}

These bounds are further improved in the setting of Corollary \ref{cor:densecover}, when $\{S_1, \cdots, S_M\}$ need not cover all pairs of points in $X$, but all pairs of points in \emph{some} $\delta$-GH approximation $X'$ of $X$, as there are typically many such approximates with many fewer points than $X$.

\section{Applications}
\label{sec:applications}
Distributed persistence has myriad applications in machine learning and data analysis, in that it can be applied in many of the same settings as standard persistent homology. We list here a few noteworthy examples. 

\begin{itemize}
	\item (Dimensionality Reduction) Given a high-dimensional data set, the goal of dimensionality reduction is to embed it in lower-dimensional space while preserving its shape. We can force the embedding to preserve the topology of the data by computing a loss comparing the persistence diagrams of many random subsets in the high-dimensional space and in the embedding.
	
	\item (Shape Registration) Given two embedded point clouds $X$ and $Y$ modeling the same shape, it can be of interest to learn a map $f: X \to Y$ aligning corresponding points. Using distributed topology, we can ask for $f$ to preserve the persistence diagrams of many random small subsets of $X$.

	\item (Feature Extraction) Given a metric space $X$, we can compute the persistence diagrams of many random small subsets of $X$, and, throwing away the subset labelings, obtain a bag-of-persistence-diagrams feature. These features can then be used in machine learning applications.
\end{itemize}

\section{Experiments}
\label{sec:experiments}

The goal of the experiments below is to corroborate the theoretical results in this paper by demonstrating that a loss function built on distributed persistence alone, and sampled on a small number of random subsets, suffices to reconstruct simple metric spaces. Suppose $X$ and $Y$ are finite subsets of Euclidean spaces and $\phi: X \to Y$ is a surjection. Theorem \ref{thm:invstab} shows that we may test if $\phi$ is a quasi-isometry by evaluating $d_{B}(\lambda^{m}(S),\lambda^{m}(\phi(S)))$ for a certain collection of subsets $S \subseteq X$. If $X$ is fixed and $Y$ is variable, we can minimize $d_{B}(\lambda^{m}(S),\lambda^{m}(\phi(S)))$ thanks to the differentiability of persistence computations; this has the effect of bringing $Y$ closer in alignment with $X$. In the following two synthetic experiments, we follow the methodology described above for $X$ as (1) $100$ points evenly distributed on a circle in $\mathbb{R}^2$ and (2) $256$ points evenly distributed on a torus in $\mathbb{R}^3$. The codomain $Y$ is initialized to be $X$ with independent Gaussian noise added coordinate-wise.  Our aim is to see whether minimizing a distributed topological functional via gradient descent succeeds in correcting for the large geometric distortion of adding Gaussian noise. In both cases, every iteration step consists of uniformly sampling $k=25$ points, denoted $S$, from $X$ and taking a step (i.e. perturbing $Y$) to minimize the loss $W_2^2(D_0(S), D_0(\phi(S))) + W_2^2(D_1(S), D_1(\phi(S)))$, where $D_i$ is the degree $i$ persistence diagram of the Rips filtration. Because we are updating $Y$ based on only a single sample $S$, we use the Adam optimizer \cite{kingma2014adam} to benefit from momentum. The results of these two experiments can be found in Figure \ref{fig:experiment}, with the first row showing the circle experiment and the second row the torus experiment. For the first (resp. second) row, the first column shows the initial state of $Y$, and the following columns show $Y$ after successive multiples of $2^{11}$ (resp. $2^{15}$) iterations. For both experiments, we observe the codomain space $Y$ re-organizing itself to closely resemble $X$. The coloring of the points in Figure \ref{fig:experiment} denotes their labeling in $X$, so that points with similar colors are nearby in $X$. The fact that the color gradients in the final positions of $Y$ are largely continuous affirm that our optimization fixes not only the global geometry of $Y$, but also the labeled pairwise distances, and hence gives a space quasi-isometric to $X$. The code used to generate these experiments is available at \href{https://github.com/aywagner/DIPOLE}{https://github.com/aywagner/DIPOLE}.


\begin{figure}[htb!]
	\centering
	\includegraphics[width=1\textwidth]{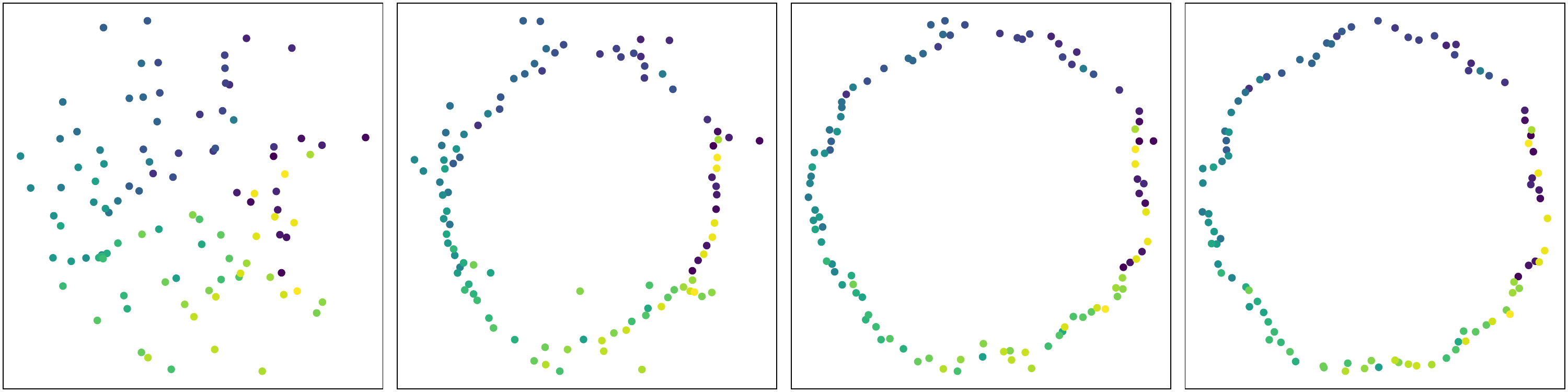}
	\includegraphics[width=1\textwidth]{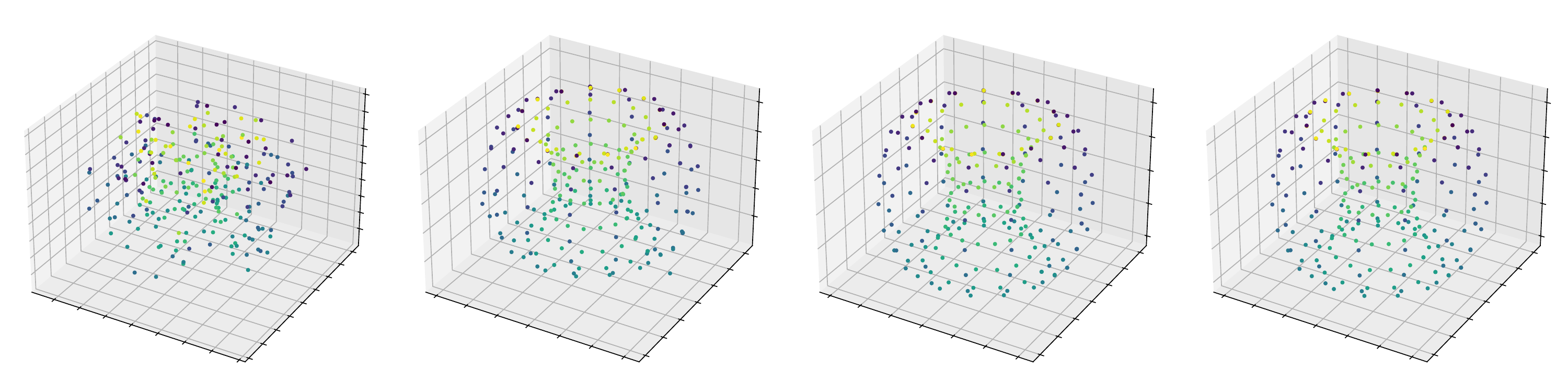}
	\caption{Synthetic optimization experiments. Columns correspond to initial, intermediate, and final positions of $Y$. Color denotes labeling.}
	\label{fig:experiment}
\end{figure}

These experiments are a proof-of-concept but can be developed into a full pipeline for dimensionality reduction. That line of investigation is beyond the scope of this paper, and was carried out by the authors in a separate paper, cf. \cite{wagner2021improving}. 
A key insight in \cite{wagner2021improving} is that adding a local metric term to the topological loss results in dramatically faster convergence to high-quality embeddings. 

\section{Conclusion}

It has long been understood that computational complexity and sensitivity to outliers are major challenges in the application of persistent homology in data analysis. Moreover, the lack of a stable inverse makes it hard to interpret which geometric information is retained in the persistence diagram, and which is forgotten. Multiple lines of research have sought to address these problems by constructing more sophisticated topological invariants and tools, such as the persistent homology transform, multiparameter persistence, distributed persistence calculations \cite{10.1145/3330345.3332147}, and discrete Morse theory. However, any gains in invertibility are compromised by sizeable increases in computational complexity.

The focus of this paper was the simplest scheme for speeding up persistence calculations: subsampling. Subsampling and bootstrapping are ubiquitous in machine learning and are already being applied in topological data analysis. What we have shown is that this simple approach also enjoys uniquely strong theoretical guarantees. In particular, the manner in which distributed persistence interpolates between geometry and topology is explicitly given by quadratic bounds. Moreover, these theoretical guarantees are complemented by the success that subsampling has seen in the TDA literature, and the robust synthetic experiments shown above.

There remain a number of outstanding problems, both theoretical and computational, that would complement the results of this paper and facilitate its practical application.

\begin{itemize}
	\item Distributed persistence, as we have defined it, consists of pairs of subsets and persistence diagrams. In many applications, we may wish to take only the persistence diagrams and forget the subset labels. What injectivity results can be obtained in this unstructured setting?
	\item Individual persistence diagrams can be challenging to work with, due to the fact that the space of diagrams admits no Hilbert space structure \cite{MR3968607, Bubenik:2020aa, 2019arXiv191013935W}, though there are a number of effective vectorizations in the literature. How can these be extended or adapted to provide vectorizations of sets of persistence diagrams coming from subsamples of a fixed point cloud? This is a more structured problem than working with arbitrary collections of persistence diagrams.
	\item If we are interested in recovering the global topology of Euclidean point clouds rather than their quasi-isometry or Gromov-Hausdorff type, it suffices to estimate pairwise distances between points in adjacent Voronoi cells, at least when working with the full Rips or \v{C}ech complex and not a skeleton. A careful analysis of this setting could dramatically decrease the Lipschitz constants appearing in Theorem \ref{thm:invstab}.
\end{itemize}

\bibliography{disttopbib}

\appendix

\section{Proof of Lemma \ref{lem:roundinglemma}}
\label{sec:roundinglemmaproof}
\begin{proof}
	The proof is a recursive construction. The first step is to add $p_1$ to $R$. We then repeat the following argument, iterating through $P$. Consider $p_n$, and let $r_*$ be the largest element of $R$ so far. If $p_n < r_* + 2\epsilon + 4\delta$, skip $p_n$. Otherwise, initialize $r_n = p_n$, and iterate over all $i < n$ and check that $p_i > (r_n + r_{*})/2$ iff $q_i > (r_n + r_{*})/2$. Every time an index $i$ is found for which this condition is violated, increment $r_n \leftarrow r_n + 2d_i$. The effect of this incrementation is to force both $q_i$ and $p_i$ to be strictly closer to $r_{*}$ than they are to $r_n$. This condition can be violated at most once for each $p_i$, hence the total sum of the incrementation is $2\delta$, at the end of which $r_n$ is added to $R$.\\
	
	Let us see why the resulting set $R$ satisfies (1) and (2). If $r_n$ was added to $R$, then it is at most $2\delta$ from $p_n$ and $2\delta + \epsilon$ from $q_n$, whereas $|r_{*} - p_n| > 2\epsilon + 4\delta$ and $|r_{*} - q_n| > \epsilon + 4\delta$ by the triangle inequality. Thus $\pi(q_n) = \pi(p_n) = r_n$. For $i<n$, the recursive incrementation ensures $\pi(p_i) = r_n$ if and only if $\pi(q_i) = r_n$, and otherwise the value of $\pi$ on $(p_i,q_i)$ is unchanged. Thus (1) is preserved. To check (2), note that if $\pi(p_i) = \pi(q_i) = r_n$ for $i < n$, then $p_i$ and $q_i$ are closer to $r_n$ than any other element in $R$. By recursive hypothesis, this distance is at most $3\epsilon + 4\delta$, so $|p_{i} - r_n|$ and $|q_{i} - r_{n}| \leq 3\epsilon + 4\delta$.\\
	
	If, on the other hand, no point was added to $R$, then $p_n < r_{*} + 2\epsilon + 4\delta$. Let $p_* \in P$ be the point corresponding to $r_*$. Since $r_{*} + 2\epsilon + 4\delta > p_{n} \geq p_{*} \geq r_{*} - 2\delta$, we know $|p_n - r_*| \leq 2\epsilon + 4\delta$ and $|q_n - r_*| \leq |q_n - p_n| + |p_n - r_*| \leq 3\epsilon + 4\delta$. If we can show that $\pi(p_n) = r_*$ and $\pi(q_n) = r_*$, the proof will be complete. If $p_n \geq r_*$ then it is clear that $\pi(p_n) = r_*$, and similarly, if $q_n \geq r_*$, we have $\pi(q_n) = r_*$. Thus we need to consider what happens if $p_n$ or $q_n$ are strictly less than $r_{*}$.\\
	
	Let $r_{**} < r_{*}$ be the penultimate point in $R$. Our goal is to show that $p_n$ or $q_n$ are strictly closer to $r_{*}$ than they are to $r_{**}$. Recall the point $p_{*} \in P$ corresponding to $r_{*}$. Since $p_{*} \leq p_n$ and $|r_* - p_*| \leq 2\delta$, we know that $p_{n} \geq r_* - 2\delta$ and $q_{n} \geq r_{*} - 2\delta - \epsilon$. Thus if $p_{n}$ or $q_n$ are strictly less than $r_{*}$, they are no further than $2\delta$ and $2\delta + \epsilon$ away, respectively. However, since $|r_{*} - r_{**}| \geq 2\epsilon + 4\delta$, the triangle inequality implies that $|p_{n} - r_{**}| \geq 2\epsilon + 2\delta$ and $|q_{n} - r_{**}| \geq \epsilon + 2\delta$. Thus, if $p_n$ or $q_n$ are smaller than $r_{*}$, they must still round up $r_{*}$ than $r_{**}$, and not $r_{**}$ or any other element of $R$.
\end{proof}

\begin{corollary}
	\label{cor:rounding}
	We can extend the set $R$ in the Rounding Lemma to a $14\delta$-dense subset $R' \subset \mathbb{R}$, without changing $\pi$ on $P \cup Q$. All that is necessary is to enrich $R$ by adding points in $(\cup_{r \in R} N(r,14\delta))^{C}$.
	\label{cor:roundingcor}
\end{corollary}

\section{Proof of Theorem \ref{thm:invstab}}
\label{sec:majorproof}

\begin{proof}
	Let $(x_1,x_2)$ be an edge in $X$, and let $(y_1,y_2)$ be the corresponding edge in $Y$, i.e. $y_1 = \phi(x_1)$ and $y_2 = \phi(x_2)$. Let $S \in \Gamma$ be a subset of size $k$ containing $(x_1,x_2)$. Let $A(S)$ be the set of appearance times of simplices in the $m$-skeleton of $S$, and define $A(\phi(S))$ similarly. Consider the following multiset of $2$-tuples: 
	\[C = \{(l,l+2\epsilon), (l,l-2\epsilon) \mid l \in A(S) \cup A(\phi(S))\},\]
	and let $P = \pi_{1}(C)$ and $Q = \pi_{2}(C)$ be the multi-sets obtained by projecting on to the first and second coordinates, respectively. By permuting the order of $P$ (and $Q$ correspondingly, to preserve the pairing), we may assume WLOG that $P$ is nondecreasing. In the notation of the Rounding Lemma, we have $\delta = \sum d_{i} = 4\epsilon|S(A)| + 4\epsilon|S(\phi(A))|$. Let $R$ be the subset given by the Rounding Lemma and Corollary \ref{cor:roundingcor}, and let  $\lambda^{R}$ denote the invariant $\lambda^{m}$ with filtration rounded to $R$. Note that if $S' \subset S$ has the property that $d_{B}(\lambda^{m}(S'),\lambda^{m}(\phi(S'))) \leq \epsilon$, then $\lambda^{R}(S') = \lambda^{R}(\phi(S'))$. To see why this is the case, let $p = (a,b) \in \lambda^{m}(S') \cup \Delta$ and $p' = (a',b') \in \lambda^{m}(\phi(S')) \cup \Delta$ be dots paired in an optimal Bottleneck matching, where $\Delta$ is the diagonal. We now show that $p$ and $p'$ are rounded to the same point.
	
	Let us first assume that $p$ is on the diagonal, so that $|b'-a'| \leq 2\epsilon$. If $p'$ is also on the diagonal, then both $p$ and $p'$ remain on the diagonal after rounding to $R$ (or, indeed, rounding to any set of values). If $p'$ is not on the diagonal, $a',b' \in A(\phi(S))$; since $|b'-a'| \leq 2\epsilon$, $a'$ are $b'$ are rounded to the same point in $R$, and hence the point $(a',b')$ is rounded to the diagonal. In either case, both $p$ and $p'$ are on the diagonal after rounding, and there is no cost in pairing them. If $p$ is not on the diagonal, then $a,b \in A(S)$, and since $a' \in [a-\epsilon,a+\epsilon]$ and $b' \in [b-\epsilon,b+\epsilon]$, we can conclude that $a$ and $a'$ round to the same point in $R$, and the same is true for $b$ and $b'$. In any case, the points $p$ and $p'$ become identical after rounding to $R$.
	
	Thus, using $\lambda^{R}$, $\phi$ preserves persistence diagrams of all subsets of $S$ of size $k$ through $k-m-1$. Let us write $\chi^{R}$ for the corresponding Euler curve invariant. Since the persistence diagram determines the Euler curve, we know that $\phi$ preserves $\chi^{R}$ for all subsets of $S$ of size $k$ through $k-m-1$.  We would like to show that $\phi$ preserves $\chi^{R}$ for all subsets of size $k-m-2$, and so $k-m-3$, etc., until one hits subsets of size two, as in Corollary \ref{cor:eulertrick}. To do so, we take a subset $P \subset X$ of size $k-m-2$, and augment it with $m+2$ additional points to obtain subsets $W$ and $W_i$, as in the proof of Lemma \ref{lem:euler}. Defining $R = K^{m}(W)$ and $T_i = K^{m}(W_i)$, Lemma \ref{lem:ietrick} provides for every Rips/\v{C}ech filtration parameter $r$:
	\begin{align*}
		(-1)^{n} \chi^{R}(K^{m}(P)(r))
		&= \sum_{i} \chi^{R}(T_i^{m}(r)) \\
		&- \sum_{i < j}\chi^{R}(T_i^{m}(r) \cap T_j^{m}(r)) \\
		&+ \sum_{i < j < k}\chi^{R}(T_i^{m}(r) \cap T_j^{m}(r) \cap T_k^{m}(r)) \\
		& \dots \\
		&- \chi^{R}(R^{m}(r)).
	\end{align*}
	Note that $\phi(R) = K^{m}(\phi(W))$ and $\phi(T_i) = K^{m}(\phi(W_i))$, and that if every simplex in $R^{m}$ lies in some $T_{i}^{m}$, then every simplex in $\phi(R)^{m}$ lies in some $\phi(T_i)^{m}$. Thus we can apply Lemma \ref{lem:ietrick} in $Y$, to deduce:
	\begin{align*}
		(-1)^{n} \chi^{R}(K^{m}(\phi(P))(r))
		&= \sum_{i} \chi^{R}(\phi(T_i)^{m}(r)) \\
		&- \sum_{i < j}\chi^{R}(\phi(T_i)^{m}(r) \cap \phi(T_j)^{m}(r)) \\
		&+ \sum_{i < j < k}\chi^{R}(\phi(T_i)^{m}(r) \cap \phi(T_j)^{m}(r) \cap \phi(T_k)^{m}(r)) \\
		& \dots \\
		&- \chi^{R}(\phi(R)^{m}(r)).
	\end{align*}
	
	Since $\phi$ preserves $\chi^{R}$ for all subsets of $S$ of size $k$ through $k-m-1$, the right-hand sides of these two equations are the same. Thus, their left sides must also be equal, so that $P$ and $\phi(P)$ have the same $m$-skeleton Euler curve $\chi^{R}$. Repeating this trick, we work our way down to subsets of size two, deducing that $\lambda^{R}((x_1,x_2)) = \lambda^{R}((y_1,y_2))$\footnote{Noting that for subsets of size two, Euler curves and persistence diagrams contain identical information.}. As $R$ is $(4 \times 14)\epsilon|S(A)| + (4 \times 14)\epsilon|S(\phi(A))|$ dense in $\mathbb{R}$, persistence stability implies that $\lambda^{m}$ and $\lambda^{R}$ are within $56\epsilon(|S(A)| + |S(\phi(A))|)$ of each other in Bottleneck distance. The triangle inequality then tells us that $d_{B}(\lambda^{m}(x_1,x_2),\lambda^{m}((y_1,y_2))) \leq 112\epsilon(|S(A)| + |S(\phi(A))|)$,  which is equivalent to $| \|x_1 - x_2\| - \|y_1 - y_2\|| \leq 112\epsilon(|S(A)| + |S(\phi(A))|)$. To conclude the proof, note that for the Rips complex, $|S(A)|, |S(\phi(A))| \leq {k \choose 2} = \frac{k^2 - k}{2} \leq \frac{k^2}{2}$, as all appearance times of simplices are just pairwise distances between points. For the \v{C}ech complex, there may be a total of $S(k,m)$ distinct appearance times in $S(A)$ or $S(\phi(A))$, one for each simplex of dimension between $1$ and $m$, that need to be rounded correctly (all dimension zero simplices necessarily appear at height zero).
	%
	%
\end{proof}

\section{Proof of Theorem \ref{thm:sparseinv}}
\label{sec:sparseproof}	

 Before proving our quasi-isometry bound, we need the following corollary of the Rounding Lemma.

\begin{corollary}
	\label{cor:roundingw1}
	Given $A_{1} \cdots A_{n}$ and $B_{1} \cdots B_{n}$ persistence diagrams, with $W^{1}(A_i,B_i) \leq \delta$, there exists a $28n\delta$-dense subset $R \subset \mathbb{R}$ such that rounding all the persistence diagrams to the grid $R \times R$ forces $\pi(A_i) = \pi(B_i)$ for all $i$.
\end{corollary}
\begin{proof}
	This is a straightforward application of the Rounding Lemma. We take the set $P$ to consist of all the birth and death times of all the dots in the $A_i$, and construct $Q$ from the $B_i$ similarly. As each $(A_i,B_i)$ pair contributes two sets of points, births and deaths, the total $\ell^{1}$ norm of pairing $P$ with $Q$ is $2 \times n \delta = 2n\delta$. By Corollary \ref{cor:rounding}, one can find a subset $R$ of density $28n\delta$ which ensures $\pi(p_i) = \pi(q_i)$ for all matched pairs $p_i \in P, q_i \in Q$, and hence $\pi(A_i) = \pi(B_i)$ for all $i$. 
\end{proof}

	\begin{proof}[Proof of Theorem \ref{thm:sparseinv}]
		Let $x_1,x_2$ be a pair of points in $X$. Without loss of generality, we can assume that at least one of these points is not in $L$, as the proof is otherwise trivial. Thus, we can extend $x_1,x_2$ to a subset $S$ of size $k$ by adding points in $L$. $S$ has $k$ subsets of size $(k-1)$. Corollary \ref{cor:roundingw1} tells us that we can find a $28(k+1)\epsilon_1$-dense subset $R \subset \mathbb{R}$ such that $\lambda^{R}(S) = \lambda^{R}(\phi(S))$, and $\lambda^{R}(S') = \lambda^{R}(\phi(S'))$ for any subset $S' \subset S$ with $|S| = (k-1)$. We can further demand from the Rounding Lemma and Corollary \ref{cor:rounding} that the appearance time of every edge in $L$ and its corresponding edge in $\phi(L)$ be exactly the same, where $R$ will now be $28(k+1)\epsilon_1 + 14\epsilon_2$ dense in $\mathbb{R}$. As in the proof of Theorem \ref{thm:invstab}, let $\chi^{R}$ be the Euler curve invariant corresponding to $\lambda^{R}$.\\
		
		The remainder of the proof is substantially similar to that of Theorem \ref{thm:invstab}, but in order to apply Lemma \ref{lem:ietrick} we must show that $\chi^{R}(S') = \chi^{R}(\phi(S'))$ for all $P \subset S$ of cardinality $k-2$ containing $x_1,x_2$. The analysis conducted in the proof of Theorem \ref{thm:invstab} shows that $\phi$ preserves $\chi^{R}$ for all subsets $S' \subseteq S$ of cardinality $k$ and $k-1$. We know that $S \setminus P$ consists of two points $z_1,z_2 \in L$. Let $T = K^{1}(S)$, $T_1 = K^{1}(P \cup \{z_1\})$, $T_2 = K^{1}(P \cup \{z_2\})$, and $W = T_1 \cap T_2 = K^{1}(P)$. Then $\phi(T) = K^{1}(\phi(S))$, $\phi(T_1) = K^{1}(\phi(P \cup \{z_1\}))$, $\phi(T_2) = K^{1}(\phi(P \cup \{z_2\}))$, and $\phi(W) = \phi(T_1) \cap \phi(T_2) = K^{1}(\phi(P))$. Finally, let $E$ be the edge $\{z_1,z_2\}$. The inclusion-exclusion principle of the Euler characteristic says that, for every filtration parameter $r$:
		\begin{align*}
			\chi^{R}(T(r))  &= \chi^{R}(T_1(r) \cup T_2(r) \cup E(r))\\ & = \chi^{R}(T_1(r)) + \chi^{R}(T_2(r)) + \chi^{R}(E(r)) - \chi^{R}(W(r))
		\end{align*} 
	
	\begin{align*}
	\chi^{R}(\phi(T)(r)) &= \chi^{R}(\phi(T_1)(r) \cup \phi(T_2)(r) \cup \phi(E)(r))\\ & = \chi^{R}(\phi(T_1(r))) + \chi^{R}(\phi(T_2(r))) + \chi^{R}(\phi(E(r))) - \chi^{R}(\phi(W(r)))	
	\end{align*}
	
	We know that $\phi$ preserves the Euler curves of $T,T_1,T_2$, and $E$, with the latter following from our deliberate choice to round the appearance times of $E$ and $\phi(E)$ to the same value. Thus the Euler curves for $W = K^{1}(P)$ and $\phi(W) =K^{1}(\phi(P))$ are the same. We can now apply Lemma \ref{lem:ietrick} and complete the proof as with Theorem \ref{thm:invstab}. 
\end{proof}

	\begin{remark}
		The above proof does not require all pairwise distances in $L$, as the inclusion-exclusion trick can be carried out with $O(k)$ intersections, rather than the full sublattice of $O(k^2)$ intersections. We have omitted this analysis as it obfuscates the statement of the theorem and does not significantly improve it.
	\end{remark}

\section{Proofs for Section \ref{sec:prob}}
\label{sec:probproof}

\subsection{Proof of Proposition \ref{prop:probcover}}

\begin{proof}
	\begin{align}
		P(A) & = 1 - P(\exists (x_1,\cdots, x_p) \mbox{ not in any } S_i)\\
		&\geq 1 - \sum_{(x_1,\cdots, x_p)} P((x_1,\cdots, x_p) \mbox{ not in any } S_i)\\
		& = 1 - {n \choose p}P((x_1,\cdots, x_p) \mbox{ not in any } S_i)\\
		& = 1 - {n \choose p}\prod_{i=1}^{M}P((x_1,\cdots, x_p) \mbox{ not in } S_i)\\
		& = 1 - {n \choose p}\prod_{i=1}^{M}(1 - P((x_1,\cdots, x_p) \subseteq S_i)).
	\end{align}
	
	An elementary counting argument provides:
	\[P((x_1,\cdots, x_p) \subseteq S_i) = \frac{{n-p \choose k-p }}{{n \choose k}}.\]
	
	Note further that:
	\[ \frac{{n-p \choose k-p }}{{n \choose k}} = \frac{k(k-1)(k-2)\cdots (k-p+1)}{n(n-1)(n-2)\cdots (n-p+1)} \geq \left(\frac{k-p+1}{n-p+1}\right)^{p}.\]
	
	Finally, observe that the effect of replacing $P((x_1,\cdots, x_p) \subseteq S_i)$ with $ \left(\frac{k-p+1}{n-p+1}\right)^{p}$ is to decrease the value of (5), and so the result is proved.
\end{proof}

\subsection{Proof of Proposition \ref{prop:probeps}}
\begin{proof}
	Our goal is to have:
	\[\epsilon \geq 1 - {n \choose p}\left(1 - \left(\frac{k-p+1}{n-p+1}\right)^{p}\right)^{M}\]
	
	which is equivalent to
	\[{n \choose p}\left(1 - \left(\frac{k-p+1}{n-p+1}\right)^{p}\right)^{M} \geq 1 - \epsilon .\]	
	
	Taking the log of both sides gives
	\[\log {n \choose p} + M \log \left(1 - \left(\frac{k-p+1}{n-p+1}\right)^{p}\right) \geq \log (1 - \epsilon) .\]
	
	Solving for $M$ gives:
	\begin{equation}
		M \geq \frac{\log(1 - \epsilon) - \log {n \choose p}}{\log \left(1 - \left(\frac{k-p+1}{n-p+1}\right)^{p}\right)}.
	\end{equation}
	
	The denominator on the right-hand side of (6) is negative, so using the identity ${n \choose p } < \left(\frac{ne}{p}\right)^p$, we can replace (6) with the strictly stronger inequality:
	\begin{equation}
		M \geq \frac{\log(1 - \epsilon) - p\log \frac{ne}{p}}{\log \left(1 - \left(\frac{k-p+1}{n-p+1}\right)^{p}\right)}.
	\end{equation}
	
	We can then apply the identity  $0 \geq -x \geq \log(1-x)$ for $x \in (0,1)$, and so replace (7) with the stronger inequality,
	\begin{equation}
		M \geq \frac{\log(1 - \epsilon) - p\log \frac{ne}{p}}{ - \left(\frac{k-p+1}{n-p+1}\right)^{p}}.
	\end{equation}
	
	The result then follows via simple algebra.
\end{proof}

\end{document}